\documentclass[a4paper]{amsart}

\usepackage{amssymb}
\usepackage{stmaryrd}

\swapnumbers

\makeatletter
\def\section{\@startsection{section}{1}%
  \z@{.7\linespacing\@plus\linespacing}{.5\linespacing}%
  {\normalfont\bfseries\centering}}
\def\@secnumfont{\bfseries}
\makeatother

\renewcommand{\o}{\circ}
\def\frak{\mathfrak}
\def\Bbb{\mathbb}
\def\Cal{\mathcal}
\def\sideremark#1{\ifvmode\leavevmode\fi\vadjust{\vbox to0pt{\vss%
  \hbox to 0pt{\hskip\hsize\hskip1em%
  \vbox{\hsize3cm\tiny\raggedright\pretolerance10000%
  \noindent #1\hfill}\hss}\vbox to8pt{\vfil}\vss}}}%

\newcommand{\al}{\alpha}

\newcommand{\ka}{\kappa}

\newcommand{\om}{\omega}
\newcommand{\ph}{\varphi}

\newcommand{\si}{\sigma}
\renewcommand{\th}{\theta}
\newcommand{\ze}{\zeta}

\newcommand{\Ga}{\Gamma}
\newcommand{\La}{\Lambda}
\newcommand{\Ph}{\Phi}

\newcommand{\Om}{\Omega}

\newcommand{\Y}{\Upsilon}

\newcommand{\Ad}{\operatorname{Ad}}

\renewcommand{\exp}{\operatorname{exp}}

\renewcommand{\o}{\circ}

\newcommand{\pmat}[1]{\begin{pmatrix}#1\end{pmatrix}}
\newcommand{\smat}[1]{\left(\begin{smallmatrix}#1\end{smallmatrix}\right)}

\let\[=\llbracket
\let\]=\rrbracket
\let\x=\times

\def\g{\frak g}
\newcommand{\fg}{{\frak g}}

\def\p{\frak p}
\def\q{\frak q}

\def\X{\frak X}
\def\({\big(}
\def\){\big)}
\def\R{\Bbb R}

\def\G{{\Cal G}}
\newcommand{\tcg}{{\tilde{\Cal G}}}
\newcommand{\tg}{{\tilde{\frak g}}}
\newcommand{\tp}{{\tilde{\frak p}}}
\newcommand{\tq}{{\tilde{\frak q}}}
\newcommand{\hg}{{\hat{\frak g}}}
\newcommand{\hp}{{\hat{\frak p}}}

\def\L{{\Cal L}}
\def\.{\hbox to5pt{\hss$\cdot$\hss}}

\newtheorem*{prop*}{Proposition \thesubsection}

\newtheorem*{thm*}{Theorem \thesubsection}

\newtheorem*{lem*}{Lemma \thesubsection}

\newtheorem*{cor*}{Corollary \thesubsection}
\theoremstyle{remark}
\newtheorem*{remark*}{Remark \thesubsection}

\begin{document}

\title{Contact projective structures and chains\\} \author{Andreas \v
  Cap\\ Vojt\v ech \v Z\'adn\'\i k} \thanks{First author supported by
  project P19500--N13 of the Fonds zur F\"orderung der
  wissenschaftlichen Forschung (FWF). Second author supported by the grant
  201/06/P379 of the Czech Science Foundation (GA\v CR).}  

\address{A.\v C: Fakult\"at f\"ur Mathematik, Universit\"at Wien,
  Nordbergstra\ss e 15, A--1090 Wien, Austria and International Erwin
  Schr\"odinger Institute for Mathematical Physics, Boltzmanngasse 9,
  A--1090 Wien, Austria\newline V.\v Z: Faculty of Education, Masaryk University,
  Po\v r\'\i\v c\'\i{} 31, 60300 Brno, Czech Republic}
\email{Andreas.Cap@univie.ac.at, zadnik@math.muni.cz}

\begin{abstract}
  Contact projective structures have been profoundly studied by
  D.J.F.~Fox. He associated to a contact projective structure a
  canonical projective structure on the same manifold. We interpret
  Fox' construction in terms of the equivalent parabolic (Cartan)
  geometries, showing that it is an analog of Fefferman's construction
  of a conformal structure associated to a CR structure. We show that,
  on the level of Cartan connections, this Fefferman--type
  construction is compatible with normality if and only if the initial
  structure has vanishing contact torsion.  This leads to a geometric
  description of the paths that have to be added to the contact
  geodesics of a contact projective structure in order to obtain the
  subordinate projective structure. They are exactly the chains
  associated to the contact projective structure, which are analogs of
  the Chern--Moser chains in CR geometry.  Finally, we analyze the
  consequences for the geometry of chains and prove that a
  chain--preserving contactomorphism must be a morphism of contact
  projective structures.
\end{abstract}

\date{October 15, 2008}
\maketitle


\section{Introduction}\label{1}
Classical projective structures can be viewed as describing the
geometry of geodesics of affine connections, viewed as unparametrized
curves (paths). The study of these structures was a very active part
of differential geometry in the first decades of the 20th century.
After some time of less activity, the interest in these geometries has
been revived during the last years. Much of this recent interest is
related to the fact that they form a simple instance of the large
class of so--called parabolic geometries.

Among the parabolic geometries there is also a contact analog of
classical projective structures, called contact projective structures.
Such a structure is given by a contact structure and a family of paths
in directions tangent to the contact distribution, which can be
realized as geodesics of some affine connection. While basic ideas on
these structures can be traced back to the classical era, they have
been formally introduced and thoroughly studied by D.J.F.~Fox in
\cite{Fox}. One of the main results in that article is that any
contact projective structure can be canonically extended to a
projective structure on the same manifold.

Studying Fox' canonical projective structure is the main purpose of
this article. We first review some fundamental facts on projective and
contact projective structures in Section \ref{2}. In Section \ref{3},
we give a geometric description of the paths in directions transverse
to the contact distribution that have to be added to the given paths
in contact directions in order to obtain the canonical projective
structure. To describe these curves, recall that for CR manifolds of
hypersurface type, there are the so--called Chern--Moser chains
introduced in \cite{CM}. They form a family of canonical
unparametrized curves available in all directions transverse to the
contact distribution. The description of Chern--Moser chains via the
canonical Cartan connection associated to a CR structure easily
generalizes to all parabolic contact structures, see \cite{C-S-Z}. In
this way, one obtains a family of chains associated to any contact
projective structure, and these are the curves to be added in order to
get the canonical projective structure, see Corollary \ref{C3.3}.

This description is obtained via another result, which is of
independent interest. In \cite{Fox}, the canonical projective
structure was obtained via the so--called ambient descriptions or cone
descriptions of contact projective and projective structures. We give
a description in terms of the canonical Cartan connections, which
shows that it is an analog of the Fefferman construction as described
in \cite{C-two}, see Proposition \ref{P3.3}. We also show that this
Fefferman type construction produces not only the canonical projective
structure but also its canonical Cartan connection if and only if the
initial contact projective structure has vanishing contact torsion, see
Theorem \ref{T3.2}.

The fact that the chains of a contact projective structure can be
realized as geodesics of an affine connection is in sharp contrast to
the cases of CR structures and Lagrangean contact structures, see
\cite{CZ}. In the latter article, we have studied chains via the path
geometry they determine. In Section \ref{4} we discuss these issues in
the contact projective case, where they are rather easy. In spite of
the fact that the contact projective structure can \textit{not} be
recovered from the path geometry of chains, we are able to prove that
contactomorphisms which preserve chains actually are morphisms of
contact projective structures, see Theorem \ref{T4.3}.

\section{Projective and contact projective structures}\label{2}
\subsection{Projective structures}\label{2.1}
A \textit{projective structure} on a smooth manifold $M$ is given by a
class of projectively equivalent linear connections $[\nabla]$ on
$TM$.  Two connections are called \textit{projectively equivalent} if
their difference tensor is of the form
$A(\xi,\eta)=\Y(\xi)\eta+\Y(\eta)\xi$ for some one--form
$\Y\in\Om^1(M)$.  Note that such a tensor is symmetric, so by
definition, projectively equivalent connections have the same torsion.
It is a well known classical result that two connections which have the
same torsion are projectively equivalent if and only if they have the
same geodesics up to parametrization.  

This means that a projective structure on $M$ is given by a class of
linear connections on $TM$ which have the same torsion and the same
unparametrized geodesics.  Since symmetrizing a connection does not
change its geodesics, it is usually assumed the connections in the
class are torsion-free, which is a natural normalization of the
structure.

We can also interpret this description as saying that a projective
structure on $M$ is given by the smooth family of paths
(unparametrized curves) formed by the geodesics. This is an example of
a so--called path geometry, i.e.~a smooth family of paths with exactly
one path through each point in each directions, see \ref{4.1} for the
precise definition. We will return to this point of view there.

The model projective structure is given by the real projective space
$\R P^{m}=\Cal P\R^{m+1}$ with the class of connections induced from
the canonical flat connection on $\R^m$. The geodesics of these
connections are the projective lines. The group of diffeomorphisms of
$\R P^{m}$ which preserve this structure is $PGL(m+1,\R)$, the
quotient of $GL(m+1,\R)$ by its center. For our purposes it is better
to work with oriented projective structures. This means replacing
$\Bbb RP^m$ by the sphere $S^m$, viewed as the space of rays in $\Bbb
R^{m+1}$. Then the appropriate group is $\tilde G:=SL(m+1,\R)$, and
the distinguished paths are the great circles on $S^m$. If $m$ is
even, then $\tilde G$ is isomorphic to $PGL(m+1,\Bbb R)$, while for
odd $m$ it is a two--fold covering.  In any case, $\tilde G$ acts
transitively both on $\R^{m+1}\setminus\{0\}$ and on $S^m$.  Let
$\tilde P\subset\tilde G$ be the stabilizer of the ray generated by
the first vector of the standard basis of $\R^{m+1}$ and let $\tilde
Q\subset\tilde P$ be the stabilizer of the vector itself, so
$S^m\cong\tilde G/\tilde P$ and $\R^{m+1}\setminus\{0\}\cong \tilde
G/\tilde Q$. In terms of matrices, $\tilde P$ is represented by block
matrices
$$
\tilde P=\left\{\pmat{\det(A)^{-1}&Z\\0&A}: A\in
  GL^+(m,\R),Z\in\R^{m*}\right\},
$$
where $GL^+(m,\R)=\{A\in GL(m,\R):\det(A)>0\}$. The subgroup $\tilde
Q\subset\tilde P$ is given by those matrices in $\tilde P$ for which
$A\in SL(m,\R)$.  $\tilde P$ is a parabolic subgroup of the simple Lie
group $\tilde G$ and the corresponding grading of the Lie algebra
$\tg=\frak{sl}(m+1,\R)$ is given by the block decomposition
$$
\pmat{\tg_0&\tg_1\\ \tg_{-1}&\tg_0}
$$
with blocks of sizes $1$ and $m$. Hence $\tg_{-1}\cong\R^m$,
$\tg_0\cong\mathfrak{gl}(m,\R)$, and $\tg_1\cong\R^{m*}$. The Lie
algebras of $\tilde P$ and $\tilde Q$ are $\tp=\tg_0\oplus\tg_1$ and
$\tilde\q=\tg_0^{ss}\oplus\tg_1$, respectively. Here $\tg_0^{ss}$
denotes the semisimple part of $\fg_0$, which is isomorphic to
$\mathfrak{sl}(m,\Bbb R)$.

General oriented projective structures admit an equivalent description
as Cartan geometries of type $(\tilde G,\tilde P)$. Projecting to the
lower right block defines a homomorphism from $\tilde P$ onto
$GL^+(m,\Bbb R)$, so we can view the latter group as a quotient of
$\tilde P$. Then with notation as above, the following holds, see
e.g.~\cite{luminy}:

\begin{thm*}\label{T2.1}
  Let $M$ be an oriented smooth manifold of dimension $\geq 2$ which
  is endowed with a projective structure. Then the oriented linear
  frame bundle of $M$ can be canonically extended to a principal
  $\tilde P$--bundle $\tilde\G\to M$, which can be endowed with a
  Cartan connection $\tilde\om\in\Om^1(\tilde\G,\tg)$. The pair
  $(\tilde\G,\tilde\om)$ is uniquely determined up to isomorphism if
  one in addition requires the curvature of $\tilde\om$ to satisfy a
  normalization condition, which will be discussed in \ref{3.2} below.
\end{thm*}

The relation between the Cartan geometry $(\tilde\G\to M,\tilde\om)$
and the projective structure can be described in several ways. On the
one hand, the oriented linear frame bundle of $M$, which we will
denote by $\tilde\G_0\to M$, can be viewed as a quotient of the Cartan
bundle $\tilde\G\to M$. Then the connection forms of the connections
in the projective class can be recovered by pulling back the
component $\tilde\om_0$ of $\tilde\om$ in $\fg_0$ along certain
sections of this quotient map. More easily, the unparametrized
geodesics of the projective structure are given by projections to $M$
of flow lines of the constant vector fields
$\tilde\om^{-1}(X)\in\X(\tilde\G)$ with $X\in\tg_{-1}$.

\subsection{Contact projective structures}\label{2.2}
Recall that a contact manifold $(M,H)$ is a smooth manifold $M$ of odd
dimension $2n+1$ together with a smooth subbundle $H\subset TM$ of
corank one which is maximally non--integrable. This means that the
bilinear bundle map $\Cal L:H\x H\to TM/H$ induced by the Lie bracket
of vector fields, which is called the \textit{Levi--bracket}, is
non--degenerate. The contact analog of projective structures was
formally introduced and thoroughly studied in \cite{Fox}. Similarly to
classical projective structures this contact analog can be described in
several equivalent ways.

The simplest description is via the analog of path geometries, for
which one only considers paths which are everywhere tangent to the
contact distribution $H$. Then a contact projective structure can be
defined as a smooth family of such contact paths with one path through
each point in each direction in $H$, which are among the geodesics of
some linear connection on $TM$. This is the definition used in
\cite{Fox}. While this implicitly also provides a definition as an
equivalence class of linear connections on $TM$, more work is needed
to obtain a nice description of this type. 

Indeed, Theorem A of \cite{Fox} provides a special class of such
connections. One starts with a contact form $\th\in\Om^1(M)$, i.e.~a
one--form whose kernel in each point $x\in M$ is the contact
distribution $H_x$. We will always assume that the contact structure
in question admits global contact forms. This amounts to the fact that
the line bundle $TM/H$ (or equivalently its dual) admit global nonzero
sections. Equivalently, these line bundles have to be orientable, and
we further assume that an orientation has been chosen. Then we can
talk about positive contact forms, and given one such form $\th$, any
other is obtained by multiplication by a positive smooth
function. Given a positive contact form $\th$, Theorem A of \cite{Fox}
shows the existence of a linear connection $\nabla$ on $TM$ which has
the given paths among its geodesics, satisfies $\nabla\th=0$ and
$\nabla d\th=0$ as well as normalization conditions on its torsion. 

Since $\nabla\th=0$, the connection $\nabla$ preserves the subbundle
$H\subset TM$, and of course the geodesics in contact directions
depend only on the restriction of $\nabla$ to a linear connection on
the vector bundle $H\to M$. Further, it turns out that all of $\nabla$
is determined by the restriction of $\nabla$ to an operator
$\Ga(H)\x\Ga(H)\to\Ga(H)$, a so--called \textit{partial connection}.
Finally, viewing the Levi--bracket $\Cal L$ as a bundle map $\La^2
H\to TM/H$, its kernel defines a corank one subbundle $\La^2_0
H\subset\La^2 H$. Since $\nabla d\th=0$, the linear connection on
$\La^2H$ induced by $\nabla$ preserves this subbundle, so $\nabla$
(respectively its restriction) is a (partial) \textit{contact
  connection}. Now parallel to the projective case, one can
define \textit{contact projective equivalence} of partial contact
connections, and characterize this in terms of the difference tensor.
This leads to a formula in terms of a smooth section $H^*$ which is
similar to the one for projectively equivalent connections, see
formula (2.8) of \cite{Fox}.

There is a significant difference to the case of projective
structures, which concerns torsion. For linear connections on the
tangent bundle, one can always remove the torsion without changing the
geodesics. This is no more true in the contact setting. By Theorem 2.1
of \cite{Fox}, the restrictions of the torsions of all the
representative connections $\nabla$ associated to contact forms as
above to $\La^2H^*\otimes H\subset\La^2T^*M\otimes TM$ coincide. This
is called the \textit{contact torsion} of the contact projective
structure. (In the setting of partial contact connections, one has to
further restrict to $(\La^2_0H)^*\otimes H$, but this needs only minor
adaptions.)

The model contact projective structure is given by the space of rays
in a symplectic vector space. Consider $\R^{2n+2}$ with the standard
linear symplectic form $\Om$. Then $\Om$ induces a contact structure
on the space $S^{2n+1}$ of rays, and the great circles tangent to the
contact subbundle (which can locally be realized as geodesics for the
standard flat connection on $\Bbb R^{2n+1}$) define a contact
projective structure. One of the main results of \cite{Fox} is the
construction of a canonical projective structure from a contact
projective structure.  For the homogeneous model, it is obvious how to
do this: One simply adds that great circles which are transverse to
the contact distribution, to obtain the homogeneous model of
projective structures. 

The contact projective structure on $S^{2n+1}$ constructed above is
evidently homogeneous under the symplectic group $G:=Sp(2n+2,\Bbb R)$.
It is easy to see that the actions of elements of $G$ are exactly
those diffeomorphisms of $S^{2n+1}$ which preserve both the contact
structure and the projective structure. Generalizing this result to
the curved case will be the main aim of Section \ref{4}. Now $G$ acts
transitively both on $\R^{2n+2}\setminus\{0\}$ and on $S^{2n+1}$, so
as homogeneous spaces $S^{2n+1}\cong G/P$ and
$\R^{2n+2}\setminus\{0\}\cong G/Q$, where $P$ is the stabilizer of the
ray generated by the first vector of the standard basis of
$\R^{2n+2}$, and $Q$ is the stabilizer of that vector. For the obvious
inclusion $G\to\tilde G$ (with $m=2n+1$) we get $P=G\cap\tilde P$ and
$Q=G\cap\tilde Q$. As in \ref{2.1}, $P$ is a parabolic subgroup in the
simple Lie group $G$. To obtain a nice presentation of the Lie algebra
$\fg$ of $G$, it is best to choose $\Om$ to be represented by the
matrix
$$
\pmat{0&0&1\\0&\Bbb J&0\\-1&0&0},
$$
with $\Bbb J=\smat{0 & \Bbb I_n\\ -\Bbb I_n &0}$ and $\Bbb I_n$
denoting identity matrix of rank $n$. Using this form, the Lie algebra
$\g=\frak{sp}(2n+2,\R)$ has the form
$$
\g=\left\{\pmat{a&U&w\\X&A&\Bbb JU^t\\z&-X^t\Bbb J&-a}\right\},
$$
with blocks of sizes $1$, $2n$ and $1$, $a,z,w\in\Bbb R$, $X\in\Bbb
R^{2n}$, $Z\in\Bbb R^{2n*}$, and $A\in\frak{sp}(2n,\Bbb R)$ (with
respect to $\Bbb J$). We obtain a grading
$\g=\g_{-2}\oplus\g_{-1}\oplus\g_0\oplus\g_1\oplus\g_2$ with
$\fg_{-2}$ corresponding to $z$, $\fg_{-1}$ to $X$, $\fg_0$ to $a$ and
$A$, and so on. By construction, $\frak p$ is formed by the matrices
which are block upper triangular, i.e.~satisfy $z=0$ and $X=0$, so
$\p=\g_0\oplus\fg_1\oplus\fg_2$. The subalgebra $\q\subset\frak p$
corresponds to those matrices, which in addition satisfy $a=0$. For
the algebra $\tg$ from \ref{2.1}, we simply obtain all matrices of the
same size, and the comparison with the description in \ref{2.1} shows
the various grading components and subalgebras.

In Theorem C of \cite{Fox}, the author proves existence of a canonical
Cartan connection associated to a contact projective structure, which
reads as follows:
\begin{thm*}                                    \label{T2.2}
  Let $(M,H)$ be a contact manifold which admits a global contact form
  and is endowed with a contact projective structure.  Then there
  exists a principal $P$--bundle $p:\Cal G\to M$ endowed with a Cartan
  connection $\om\in\Om^1(\Cal G,\g)$ such that
  $H=Tp(\om^{-1}(\g_{-1}\oplus\frak p))$ and the contact geodesics are
  projections to $M$ of flow lines of constant vector fields
  $\om^{-1}(X)\in\X(\G)$ with $X\in\g_{-1}$.  The pair $(\Cal G,\om)$ is
  uniquely determined up to isomorphism provided that one in addition
  requires the curvature of $\om$ to satisfy a normalization condition
  discussed in \ref{3.2} below.
\end{thm*}

\begin{remark*}\label{R2.2}
  (1) The normalization condition in the Theorem is a generalization
  of the uniform normalization condition for parabolic geometries. As
  we shall discuss in more detail in \ref{3.2} below, it reduces to
  the standard condition if and only if the
  projective contact structure has vanishing contact torsion, see
  Proposition 4.1 of \cite{Fox}. In this special case, the Theorem
  follows from general results on parabolic geometries, see
  \cite{C-S}.

\noindent
(2) The description of the relation between the Cartan geometry and
the underlying contact projective structure in the Theorem is
different from the original one in \cite{Fox}. The characterization in
\cite{Fox} uses the ambient connection to be discussed in \ref{2.3}
below. Section 4.3 of \cite{Fox} discussed the characterization of
contact geodesics via the development of curves (induced by the Cartan
connection $\om$) into the homogeneous model $G/P=S^{2n+1}$.  Contact
geodesics on $M$ are exactly the curves which develop to the contact
geodesics in the model. In \cite{C-S-Z} it is shown how the
description in terms of development is equivalent to being a
projection of an integral curve of a certain type of constant vector
fields (of the Cartan connection $\om$). Then it suffices to observe
that the contact geodesics in the model $G/P$ are precisely the orbits
of one-parameter subgroups of $G$ generated by elements of $\g_{-1}$.

\noindent
(3) There is another distinguished family of curves in the model
space. As in (2), they may be either characterized via development or
as projections of integral curves of constant vector fields, but this
time with generator in $\fg_{-2}$. In view of the similarity to the
concept in CR geometry induced by Chern--Moser, these are called
\textit{chains}. In particular, a chain is uniquely determined by its
initial direction as an unparametrized curve. For the homogeneous
model $S^{2n+1}$, the chains are exactly those great circles which are
transverse to the contact distribution.
\end{remark*}

\subsection{Ambient descriptions}\label{2.3}
The basis of the construction of projective structure subordinate to a
contact projective structure is the so--called ambient description or
cone description of projective and contact projective structures. In
the projective case, this goes back to the work of Tracy Thomas in the
1930's, in the contact projective case it is due to Fox. In \cite{Fox}
the ambient connection is constructed first (in Theorem B) and then
used to construct a Cartan connection. Here we take the opposite point
of view, and use the Cartan connection to construct the ambient
connection.

The starting point for the ambient description of both types of
structure is a principal bundle $\L\to M$ with structure group $\Bbb
R_+$, namely the frame bundle of the bundle of
$(\frac{-1}{m+1})$--densities. In the contact projective case, it is
easy to see that one may also view this density bundle as a square
root of the bundle of positive contact forms. In the projective case,
$\Cal L$ can be constructed from the Cartan bundle $\tcg\to M$ via a
homomorphism $\tilde P\to\Bbb R_+$ with kernel $\tilde Q\subset\tilde
P$. Hence $\L\cong\tcg/\tilde Q$, and $\tcg\to\L$ is a principal
bundle with structure group $\tilde Q$, on which $\tilde\om$ is a
Cartan connection. In particular, $T\L\cong\tcg\x_{\tilde
  Q}(\tg/\tilde{\frak q})$ with the action of $\tilde Q$ coming from
the adjoint representation. In the contact projective case, there is a
completely analogous description in terms of the canonical Cartan
geometry $(\Cal G\to M,\om)$ induced by the contact projective
structure. In particular, $T\L\cong\G\x_Q(\fg/\frak q)$ in the contact
projective case. 

\begin{prop*}
  Consider $\tilde G:=SL(m+1,\Bbb R)$ and let $\tilde Q\subset\tilde
  G$ be the stabilizer of the first vector in the standard basis of
  $\Bbb R^{m+1}$. Then, as a representation of $\tilde Q$, $\tg/\tq$ is
  isomorphic to the restriction to $\tilde Q$ of the standard
  representation $\Bbb R^{m+1}$ of $\tilde G$. 
  
  If $m$ is odd, say $m=2n+1$, then the analogous statement holds for
  $G=Sp(2n+2,\Bbb R)$ and the stabilizer $Q\subset G$ of the first
  basis vector.
\end{prop*}
\begin{proof}
  The Lie subalgebra $\tq\subset\tg$ consists of all matrices for which
  all entries in the first column are zero. To describe the
  $\tilde Q$--representation $\tg/\tq$, we may thus simply look at the action
  of the adjoint representation of $\tilde Q$ to the first column of
  matrices. Since $\tilde G$ is a matrix group, the adjoint
  representation is given by conjugation. By definition, the first
  column of any matrix in $\tilde Q$ equals the first unit vector, so
  multiplying any matrix from the right by an element of $\tilde Q$
  leaves the first column unchanged. But this implies that for
  $A\in\tilde Q$ and $X\in\tg$, the first column of $AXA^{-1}$ equals
  the first column of $AX$, which implies the claim for $\tilde Q$.
  But then the same statement is true for any subgroup of $\tilde Q$,
  hence in particular for $Q=\tilde Q\cap G$ in the case of odd $m$. 
\end{proof}

Using this, we may view the bundle $T\L\to\L$ as the associated bundle
$\tcg\x_{\tilde Q}\Bbb R^{m+1}$ respectively $\G\x_Q\Bbb R^{2n+2}$,
and since the inducing representations are restrictions of
representations of $\tilde G$ respectively $G$, we can invoke the
general construction of \cite{tractors}. This shows that the Cartan
connection $\tilde\om$ respectively $\om$ induces a linear connection
on $T\L$. 

\begin{thm*}\label{T2.3}
  (1) The linear connection on $T\L$ induced by a projective structure
  on $M$ as described above coincides with the ambient connection from
  Theorem 3.1 of \cite{Fox}.

\noindent
(2) The linear connection on $T\L$ induced by a contact projective
structure on $M$ as described above coincides with the ambient
connection from Theorem B of \cite{Fox}.
\end{thm*}
\begin{proof}
  There are various ways to prove this, which all boil down to rather
  straightforward verifications. On the one hand, one may simply verify
  that the linear connections we have constructed satisfy the
  properties listed in the theorems of \cite{Fox}, and then invoke the
  uniqueness parts of these theorems. 
  
  Even easier, one may follow the construction of the Cartan
  connection in \cite{Fox} backwards. First, one lifts the principal
  $\Bbb R_+$--action on $\L$ to a free right action by vector bundle
  homomorphisms on $T\L$ in such a way that the orbit space $T\L/\Bbb
  R_+$ (which evidently is a vector bundle over $\L/\Bbb R_+=M$) can
  be identified with the standard tractor bundle of the structure in
  question. Then one shows that the ambient connection induces a
  tractor connection on that bundle, which by the general methods of
  \cite{tractors} gives rise to a Cartan connection. This corresponds
  to the fact that $\tilde P/\tilde Q\cong\Bbb R_+$ acts on
  $\tcg/\tilde Q=\L$ with orbit space $\tcg/\tilde P=M$ (and the
  analogous statement for $P/Q$).  Now one immediately verifies that
  the lift of the action is exactly defined in such a way that the
  linear connection on $T\L$ induces the usual tractor connection on
  the tractor bundles, which completes the proof.
\end{proof}

\section{The subordinate projective structure}\label{3}
Having collected the background, we can now move to proving the first
main results of this article. We show that the construction of a
projective structure subordinate to a contact projective structure in
\cite{Fox} can be interpreted as a generalized Fefferman construction.
This interpretation leads to immediate payoff, since it implies a
geometric description of the subordinate projective structure in terms
of chains.

\subsection{The Fefferman--type construction}\label{3.1}
The scheme for generalized Fefferman constructions is by now fairly
familiar, see \cite{C-two}, where also the application to contact
projective structures was suggested. As before, consider
$G=Sp(2n+2,\Bbb R)$ and $\tilde G:=SL(2n,\Bbb R)$, let
$\psi:G\hookrightarrow\tilde G$ be the obvious inclusion. Then put
$i:=\psi|_P:P\to\tilde P$ and $\al:=\psi':\fg\to\tg$. Now suppose that
we have given a contact manifold $(M,H)$ of dimension $m=2n+1$, which
is endowed with a contact projective structure, and let $(\G\to
M,\om)$ be the canonical Cartan geometry of type $(G,P)$ determined by
this structure as in \ref{2.2}. Then the homomorphism $i:P\to\tilde P$
defines a left action of $P$ on $\tilde P$, so we can form the
associated bundle $\tcg:=\G\x_P\tilde P\to M$. This clearly is a
principal bundle with structure group $\tilde P$, and we have a
natural map $j:\G\to\tcg$ induced by mapping $u\in\G$ to the class of
$(u,e)$. It is easy to prove (compare with 3.1 and 3.2 of \cite{CZ})
that there is a unique Cartan connection $\tilde\om\in\Om^1(\tcg,\tg)$
such that $j^*\tilde\om=\al\o\om$. 

This construction actually defines a functor from Cartan geometries of
type $(G,P)$ to Cartan geometries of type $(\tilde G,\tilde P)$, both
living on same manifolds. (The fact that we obtain a geometry on the
same manifold is due to the fact that $\tilde P\cap G$ is already a
parabolic subgroup of $G$. For other Fefferman--type constructions,
this is not the case. Then one has to pass to a parabolic subgroup
containing this intersection, and the new geometry will be defined on
the total space of a natural bundle.) Since any Cartan geometry of
type $(\tilde G,\tilde P)$ on a manifold $M$ gives rise to an
underlying projective structure, we obtain a functor mapping contact
projective structures to projective structures on the same manifold.
It is not clear, however, whether the Cartan connection $\tilde\om$ is
normal and hence coincides with the canonical Cartan connection
associated to the projective structure in general. This is a familiar
phenomenon of generalized Fefferman constructions.

Before we discuss the question of normality of $\tilde\om$, we give a
geometric description of the projective structure produced by the
generalized Fefferman construction. 
\begin{prop*}\label{P3.1}
  Let $(\G\to M,\om)$ be a Cartan geometry of type $(G,P)$ and let
  $(\tcg\to M,\tilde\om)$ be the Cartan geometry of type $(\tilde
  G,\tilde P)$ obtained by the generalized Fefferman construction. 
  
  Then the paths of the projective structure determined by $(\tcg\to
  M,\tilde\om)$ are the projections of the flow lines of the constant
  vector fields $\om^{-1}(X)\in\frak X(\G)$ generated by elements
  $X\in\fg_{-1}\cup\fg_{-2}$.  
\end{prop*}
\begin{proof}
  It is well known that the paths of the projective structure can be
  obtained as the projections of the flow lines of the constant vector
  fields $\tilde\om^{-1}(\tilde X)$ for all elements $\tilde
  X\in\tg_{-1}$. Moreover, it is well known that there is exactly one
  such path through each point of $M$ in each direction. As we have
  seen above, viewing $\fg$ as a subalgebra of $\tg$, the Cartan
  connection $\tilde\om$ is characterized by $j^*\tilde\om=\om$. In
  particular, for $X\in\fg\subset\tg$, the constant vector fields
  $\om^{-1}(X)\in\frak X(\G)$ and $\tilde\om^{-1}(X)\in\frak X(\tcg)$
  are $j$--related. Hence their flows are $j$--related and in
  particular have the same projection to $M$.
  
  From the description of the Lie algebras $\g\subset\tg$ in \ref{2.1}
  and \ref{2.2}, we first see that $\g_{-2}\subset\tg_{-1}$. Hence for
  $X\in\fg_{-2}$, the projection of the flow line of $\om^{-1}(X)$ is
  among the paths of the projective structure. The tangent directions
  of these paths exhaust all directions which are transverse to the
  contact distribution.
  
  On the other hand, $\g_{-1}\subset\tg_{-1}\oplus\tg_0$ (with a
  nontrivial component in $\tg_0$ for any nonzero element of
  $\g_{-1}$). For $X\in\fg_{-1}$ let $\tilde X$ be the
  $\tg_{-1}$--component of $X$ (i.e.~the matrix with the same first
  column as $X$ and all other columns zero), and put $\tilde
  A=X-\tilde X\in\tg_0$. From the explicit presentations of $\fg$ and
  $\tg$ one immediately verifies that $[\tilde A,\tilde X]=0$, and
  hence $\Ad(\exp(-t\tilde A))(\tilde X)=\tilde X$ for all $t$. Now
  for $u\in\G$, let $\tilde c(t)$ be the flow line of the constant
  vector field $\tilde\om^{-1}(\tilde X)$. Then the curve
  $c(t):=\tilde c(t)\cdot\exp(t\tilde A)$ has the same projection to
  $M$ as $\tilde c(t)$, so this projection is among the paths
  determined by the projective structure. But denoting by $r$ the
  principal right action and by $\ze_{\tilde A}$ the fundamental
  vector field generated by $\tilde A$, one computes that
$$
c'(t)=Tr^{\exp(-t\tilde A)}\cdot\tilde c(t)+\ze_{\tilde A}(c(t)),
$$
and so $\tilde\om(c'(t))=\Ad(\exp(-t\tilde A))(\tilde X)+\tilde
A=X$ for all $t$. This shows that the flow line of $\om^{-1}(X)$ is
also among the paths of the induced projective structure. Since the
tangents of such paths exhaust all directions in the contact
distribution, this completes the proof. 
\end{proof}

\begin{remark*}\label{R3.1}
  (1) A nice alternative argument for the last part of the proof is as
  follows: Since $[\tilde A,\tilde X]=0$, we get
  $\exp(tX)=\exp(t\tilde X)\exp(t\tilde A)$, and hence the exponential
  curves generated by $X$ and $\tilde X$ have the same projection to
  $\tilde G/\tilde P$. Via development, this implies the same result
  for the flow lines of the constant vector fields. 

\noindent
(2) A regular Cartan geometry $(\G\to M,\om)$ of type $(G,P)$ as in
the proposition gives rise to a contact projective structure on $M$.
The distinguished paths (in contact directions) of this structure are
the flow lines of the vector fields $\om^{-1}(X)$ for $X\in\fg_{-1}$.
The proposition in particular says, that these are among the paths of
the projective structure obtained via the generalized Fefferman
construction. Hence the projective structure obtained from the
generalized Fefferman construction is \textit{subordinate} to the
initial contact projective structures in the sense of Definition 3.1
of \cite{Fox}.
\end{remark*}

\subsection{Normality}\label{3.2} 
As mentioned above, there are few general results on the compatibility
of Fefferman type constructions with normality of Cartan connections,
except for the fact that the result of a generalized Fefferman
construction is locally flat if and only if the original geometry is
locally flat. To discuss normality in our case, let us first recall
the normalization condition used for parabolic geometries. Consider a
semisimple Lie algebra $\fg$ with a parabolic subalgebra $\frak p$ and
the corresponding grading $\fg=\fg_-\oplus\fg_0\oplus\frak p_+$ (with
$\frak p=\fg_0\oplus\frak p_+$). Then the Killing form induces a
duality between $\fg/\frak p$ and $\frak p_+$, which is equivariant
for the natural action of any parabolic subgroup $P\subset G$ with Lie
algebra $\frak p$. Now there is a standard complex for computing the
Lie algebra homology of $\frak p_+$ with coefficients in $\fg$. The
differential in this complex is often denoted by $\partial^*$ and
referred to as the \textit{Kostant codifferential} since it can also
be obtained by dualizing a Lie algebra cohomology differential. For
the normalization condition, we need the map
$$
\partial^*:\La^2\frak p_+\otimes\fg\to \frak p_+\otimes\fg,
$$
which on decomposable elements is given by
$$
\partial^*(Z\wedge W\otimes
A)=-W\otimes[Z,A]+Z\otimes[W,A]-[Z,W]\otimes A.
$$
Now the curvature of a Cartan geometry $(\Cal G,\om)$ of type $(G,P)$
can be described by the curvature function $\ka:\Cal G\to
L(\La^2(\frak g/\frak p),\fg)$, which is characterized by 
$$
\ka(u)(X+\frak p,Y+\frak p)=d\om(\om^{-1}(X)(u),\om^{-1}(Y)(u))+[X,Y].
$$
As noted above, the target space of $\ka$ can be identified with
$\La^2\frak p_+\otimes\fg$, and the geometry is called \textit{normal}
if $\partial^*\o\ka=0$.

This is the normalization condition used for projective structures in
Theorem \ref{T2.1}. For contact projective structures, this
normalization condition is not general enough, however. The reason for
this can be seen from one of the nice properties of the normalization
condition given by the Kostant codifferential. Namely, there is an
operator $\square:\La^2\frak p_+\otimes\fg\to\La^2\frak p_+\otimes\fg$
called the \textit{Kostant Laplacian}. This is not equivariant for the
action of the parabolic subgroup $P$ but only for its Levi component,
a subgroup $G_0\subset P$ with Lie algebra $\fg_0$. Now due to the
gradings on $\frak p_+$ and $\fg$, the space $\La^2\frak
p_+\otimes\fg$ is naturally graded, so one may split the curvature
function $\ka$ into homogeneous components with respect to this
gradings. One has to assume throughout that the geometry is regular,
so all homogeneous components of degree less or equal to zero vanish
identically. If this is the case, then it turns out that the lowest
nonzero homogeneous component of $\ka$ always has values in
$\ker(\square)$. This is extremely useful, since $\ker(\square)$ can
be computed explicitly as a $G_0$--representation (which is the main
step towards the proof of Kostant's version of the Bott--Borel--Weil
theorem in \cite{Kostant}).

For the parabolic pairs $(\fg,\frak p)$ and $(\tg,\tp)$ considered in
Section \ref{2}, the description of $\ker(\square)$ is particularly
easy. In each case, this is an irreducible representation of $G_0$,
contained in one homogeneity. The result is listed in the tables
below, and $\ker(\square)$ is always the component of highest weight
in the indicated subrepresentation.

\begin{tabular}{cc}
\parbox[c][][c]{0.45\textwidth}{
\begin{center}
$(\frak g,\frak p)$, $n=1$ \\[6pt] 
\begin{tabular}{|c|c|}
  \hline homog. & \parbox[c][1.4\totalheight][c]{55pt}{contained in} \\
  \hline 3 & \parbox[c][1.4\totalheight][c]{55pt}{$\g_1\wedge\g_2\otimes\g_0$}\\
  \hline
\end{tabular}
\end{center}}&
\parbox[c][][c]{0.45\textwidth}{
\begin{center}
$(\frak g,\frak p)$, $n>1$ \\[6pt]
\begin{tabular}{|c|c|}
  \hline homog. & \parbox[c][1.4\totalheight][c]{55pt}{contained in}\\
  \hline 2 & \parbox[c][1.4\totalheight][c]{55pt}{$\g_1\wedge\g_1\otimes\g_0$} \\
  \hline
\end{tabular}
\end{center}}
\end{tabular}

\bigskip

\begin{tabular}{cc}
\parbox[c][][c]{0.45\textwidth}{
\begin{center}
$(\tg,\tp)$, $m=2$ \\[6pt] 
\begin{tabular}{|c|c|}
  \hline homog. & \parbox[c][1.4\totalheight][c]{55pt}{contained in} \\
  \hline 3 &
  \parbox[c][1.4\totalheight][c]{55pt}{$\tg_1\wedge\tg_1\otimes\tg_1$}\\
  \hline
\end{tabular}
\end{center}}&
\parbox[c][][c]{0.45\textwidth}{
\begin{center}
$(\tg,\tp)$, $m>2$ \\[6pt]
\begin{tabular}{|c|c|}
  \hline homog. & \parbox[c][1.4\totalheight][c]{55pt}{contained in}\\
  \hline 2 &
  \parbox[c][1.4\totalheight][c]{55pt}{$\tg_1\wedge\tg_1\otimes\tg_0$} \\
  \hline
\end{tabular}
\end{center}}
\end{tabular}

\bigskip

In particular, we see that in all cases the maps in $\ker(\square)$
have values in $\frak p\subset\fg$ respectively in
$\tp\subset\tg$. Since the same is evidently true for all maps of
higher homogeneous degree, we see that in both cases the curvature
function of a regular normal parabolic geometry always has values in
$\La^2\frak p_+\otimes\frak p$, i.e.~such geometries are always
torsion free. Now it is easy to see that a Cartan geometry of type
$(G,P)$ is torsion free if and only if the induced contact projective
structure has vanishing contact torsion. 

To deal with contact projective structures with non--vanishing contact
torsion, one therefore has to generalize the normalization condition,
and this has been done in \cite{Fox}. In Definition 4.1 of that
article, the author explicitly describes a $P$--submodule $\Cal
K\subset\wedge^2(\g/\p)^*\otimes\g$, which consists of maps of
positive homogeneity and contains $\ker(\partial^*)$. The
normalization condition used in Theorem \ref{T2.2} then is that the
curvature function has values in $\Cal K$. 

Having the necessary background at hand, we can now clarify
compatibility of the generalized Fefferman construction with
normality.

\begin{thm*}\label{T3.2}
  Let $(\G\to M,\om)$ be a Cartan geometry of type $(G,P)$ satisfying
  the generalized normalization condition discussed above and let
  $(\tilde\G\to M,\tilde\om)$ be the result of the generalized
  Fefferman construction from \ref{3.1}.  Then $\tilde\om$ is normal
  if and only if $\om$ is torsion free. Moreover, $\tilde\om$ is
  locally flat if and only if $\om$ is locally flat.
\end{thm*}
\begin{proof}
  The description of the generalized Fefferman construction in
  \ref{3.1} immediately leads to the relation between the curvatures
  of the two geometries. Let us denote by $\tilde\ka$ and $\ka$ the
  curvature functions of $\tilde\om$ and $\om$. Noting that
  $\al=\psi':\fg\to\tg$ is a homomorphism of Lie algebras, we obtain
  (compare with Proposition 3.3 of \cite{CZ})
\begin{equation}\label{eq}
  \tilde\ka(j(u))(\tilde X+\tp,\tilde Y+\tp)= \al(\ka(u)(X+\frak
  p,Y+\frak p)),
\end{equation}
for all $u\in\G$, $\tilde X,\tilde Y\in\tg$ and $X,Y\in\fg$ such that
$\al(X)+\tp=\tilde X+\tp$ and likewise for $Y$ and $\tilde Y$. Note
that for given $\tilde X$, we can always find an element $X$ with this
property, since $\al$ induces a linear isomorphism $\fg/\frak
p\to\tg/\tp$. Note also, that by equivariancy \eqref{eq} uniquely
determines $\tilde\ka$. Since $\al$ is injective, we see that
$\tilde\ka$ vanishes identically if and only if $\ka$ vanishes
identically, so the statement about local flatness follows readily. 

Second, $\tilde\om$ by definition is torsion free if and only if
$\tilde\ka$ has values in $\La^2(\tg/\tp)\otimes\tp$, and since
$\tp\subset\tg$ is $\tilde P$--invariant this is equivalent to the
same statement for $\tilde\ka\o j$. Since $\al^{-1}(\tp)=\frak p$ by
construction, the latter statement via \eqref{eq} is equivalent to
$\ka$ having values in $\La^2(\fg/\frak p)\otimes\frak p$ and hence to
torsion freeness of $\om$. As we have seen above, normal Cartan
connections of type $(\tilde G,\tilde P)$ are always torsion free, so
we see that normality of $\tilde\om$ implies torsion freeness of
$\om$. 

Let us conversely assume that $\om$ is torsion free and satisfies the
generalized normalization condition from Theorem \ref{2.2}. Then by Proposition
4.1 of \cite{Fox} the curvature function $\tilde\ka$ has values in
$\ker(\partial^*)$, so we may apply general results for parabolic
geometries. The isomorphism $\underline{\al}:\fg/\frak p\to\tg/\tp$
induced by $\al$ is equivariant over the inclusion
$P\hookrightarrow\tilde P$, so the same is true for
$\ph:=(\underline{\al}^{-1})^*:(\fg/\frak p)^*\to(\tg/\tp)^*$. Hence
also the map $\Ph:=\La^2\ph\otimes\al$ is equivariant in the same
sense, and in terms of this map we can write \eqref{eq} as $\ka\o
j=\Ph\o\ka$. Now let $\tilde\partial^*$ be the Kostant codifferential
associated to $(\tg,\tp)$. Then equivariancy of $\Ph$ implies that
$\Ph^{-1}(\ker(\tilde\partial^*))\subset\La^2(\fg/\frak
p)^*\otimes\fg$ is a $P$--submodule. Clearly, normality of $\tilde\om$
is equivalent to the fact that $\ka$ has values in this
$P$--submodule. By Corollary 3.2 of \cite{C-tw}, this is equivalent to
the fact that the harmonic part $\ka_H$ of the curvature function has
values in there.

As discussed above, $\ka_H(u)$ has values in
$\ker(\square)\subset\wedge^2(\g/\p)^*\otimes\g$, which is an
irreducible representation of $G_0$. Denoting by $w\in\ker(\square)$ a
highest weight vector in this representation, it is therefore
sufficient to prove that $\Phi(w)\in\ker(\tilde\partial^*)$. This can
be verified by a simple direct computation. Alternatively, it is
easy to verify that the $G_0$--representation $(\tg/\tp)^*\otimes\tg$
in which $\tilde\partial^*$ has values does not contain an irreducible
component isomorphic to $\ker(\square)$.
\end{proof}

\begin{remark*}\label{R3.2}
  The proof of this theorem is significantly easier then the proofs of
  normality for the classical Fefferman construction (see
  \cite{fefferman}) or other generalized Fefferman constructions. This
  is due to the fact that $G\cap\tilde P=P$ and hence
  $\fg\cap\tp=\frak p$ in our case. The second property directly shows
  that torsion freeness of $\om$ implies torsion freeness of
  $\tilde\om$, which otherwise needs more involved proofs. On the
  other hand, the first property implies equivariancy of the map
  $\Phi$, which together with the general results obtained using BGG
  sequences allow a reduction of the problem to harmonic curvature.
\end{remark*}

\subsection{Comparing to the construction by Fox}\label{3.3}
For a contact projective structure on a contact manifold $(M,H)$,
there is the canonical Cartan geometry $(\G\to M,\om)$ from Theorem
\ref{T2.2}. Applying to this geometry the generalized Fefferman
construction from \ref{3.1}, we obtain a canonical projective
structure on $M$, which is subordinate to the contact projective
structure in the sense of Remark \ref{R3.1}(2). Our final aim in this
section is to prove that the result coincides with the subordinate
projective structure constructed in Section 3.3 of \cite{Fox}.

Fox' construction is based on the ambient description of contact
projective and projective structures as discussed in \ref{2.3}. There
we already noticed that the spaces on which the ambient connection is
defined are the same for both types of structures. Moreover, the
ambient connection associated to a contact projective structure in Theorem B
of \cite{Fox} satisfies all the properties of a projective ambient
connection from Theorem 3.1 of \cite{Fox}, except for torsion
freeness. Symmetrizing the contact projective ambient connection, one
obtains a torsion free connection, which then is the canonical
connection associated to a projective structure. This is the canonical
projective structure as defined by Fox. Notice that the ambient
connection associated to a contact projective structure is torsion
free if and only if the structure has vanishing contact torsion. This
in turn is equivalent to the fact that this ambient connection
coincides with the ambient connection of the canonical subordinate
projective structure defined by Fox, which is the analog of Theorem
\ref{T3.2} in this setting.

\begin{prop*}\label{P3.3}
  For a contact projective structure on a contact manifold $(M,H)$,
  the subordinate projective structure obtained via the generalized
  Fefferman construction coincides with the subordinate projective
  structure constructed in Section 3.3 of \cite{Fox}.
\end{prop*}
\begin{proof}
  Let $(\G\to M,\om)$ be the Cartan geometry associated to the contact
  projective structure as in Theorem \ref{2.2} and let $(\tcg\to
  M,\tilde\om)$ be the result of the generalized Fefferman
  construction from \ref{3.1}. We want to show that, via the procedure
  from \ref{2.3}, these two Cartan geometries lead to the same ambient
  connection. From property 4 of an ambient connection in Theorem 3.1
  of \cite{Fox} one easily concludes that the paths of a projective
  structure can be realized as projections of geodesics of the ambient
  connection. Since symmetrizing the projective ambient connection
  does not change its geodesics, this will complete the proof.
  
  To compute the ambient connections, recall that we can realize the
  space $\L$ on which the ambient connection is defined as $\Cal G/Q$
  or $\tcg/\tilde Q$. Further, $\fg/\frak q\cong\Bbb R^{2n+2}$ as a
  representation of $Q$ and $\tg/\tq\cong\Bbb R^{2n+2}$ as a
  representation of $\tilde Q$. These identifications are compatible
  with the isomorphism $\fg/\frak q\cong\tg/\tq$ induced by the
  inclusion $\fg\hookrightarrow\tg$. Using $T\L\cong\G\x_Q(\fg/\frak
  q)$, vector fields on $\L$ are in bijective correspondence with
  $Q$--equivariant smooth functions $\G\to\Bbb R^{2n+2}$. Likewise,
  $T\L\cong\tcg\x_{\tilde Q}(\tg/\tq)$ identifies such vector fields
  with $\tilde Q$--equivariant smooth functions $\tcg\to\Bbb
  R^{2n+2}$. The correspondence between functions and vector fields is
  given by taking preimages under the Cartan connections, and then
  projecting to the base. For the canonical inclusion $j:\G\to\tcg$,
  we by definition have $j^*\tilde\om=\om$ (identifying $\fg$ with a
  subset of $\tg$). This immediately implies that if $\tilde
  f:\tcg\to\Bbb R^{2n+2}$ is the equivariant function corresponding to
  $\eta\in\frak X(\L)$, then the equivariant function $f:\G\to\Bbb
  R^{2n+2}$ corresponding to $\eta$ is simply given by $f=\tilde f\o
  j$.
  
  As we know from \ref{2.3}, the ambient connections are special cases
  of tractor connections, so their actions are described in terms of
  equivariant functions in the proof of Theorem 2.7 in
  \cite{tractors}. We first look at the contact projective ambient
  connection. Given another vector field $\xi\in\frak X(\L)$, we first
  have to choose a lift $\hat\xi\in\frak X(\G)$. Then the function
  $\G\to\Bbb R^{2n+2}$ corresponding to the covariant derivative of
  $\eta$ in direction $\xi$ is given by $\hat\xi\cdot f+\om(\hat\xi)\o
  f$. (In the first summand the vector field differentiates the
  function, while in the second, $\om(\hat\xi)$ acts algebraically on
  the values of $f$.) Now we can extend $Tj\o\hat\xi$ to a lift
  $\tilde\xi\in\frak X(\tcg)$ of $\xi$. From $j^*\tilde\om=\om$ we
  conclude that $\tilde\om(\tilde\xi)\o j=\om(\hat\xi)$, and by
  construction $\hat\xi\cdot f=(\tilde\xi\cdot\tilde f)\o j$. But this
  says that the function describing the covariant derivative with
  respect to the projective ambient connection is just the equivariant
  extension of the functions describing the covariant derivative with
  respect to the contact projective ambient connection. This shows
  that the two connections actually coincide, which completes the
  proof.
\end{proof}

Of course, the nice geometric interpretation of the subordinate
projective structure provided by the generalized Fefferman
construction in Proposition \ref{3.1} now carries over to the
construction by Fox. 
\begin{cor*}\label{C3.3}
  In the language of paths, the canonical subordinate projective
  structure defined in \cite{Fox} is obtained by adding the chains of
  a contact projective structure to the contact geodesics.
\end{cor*}

\section{The path geometry of chains}\label{4}
The chains in a contact projective structure determine a generalized
path geometry. For Lagrangean contact structures and partially
integrable almost CR structures, this path geometry and its relation to the
parabolic geometry associated to the original structure has been
discussed in \cite{CZ}. For contact projective structures, this
relation is much easier, since the path geometry of chains is obtained
as a restriction of the path geometry induced by the subordinate
projective structure. This is a simple instance of the general
construction of correspondence spaces from \cite{C-tw}. Therefore, the
analogs of the results form \cite{CZ} on the path geometry of chains
are rather easy to deduce. Still this path geometry turns out to be
very useful, since it allows us to prove that a contactomorphism
between two contact projective structures which maps chains to chains
actually is a morphism of the contact projective structures.

\subsection{Generalized path geometries}\label{4.1}
As we have briefly mentioned in \ref{2.1}, path geometries can be
viewed as smooth families of curves on a manifold $M$ with exactly one
curve through each point in each direction. More formally, let $M$ be
a smooth manifold of dimension $m$, and let $N:=\Cal PTM$ be the
projectivized tangent bundle of $M$. This is a smooth fiber bundle
over $M$ with fiber the projective space $\Bbb RP^{m-1}$. In
particular, there is a canonical projection $\pi:N\to M$ and we have
the vertical subbundle $VN\subset TN$. Next, by definition a point 
in $N$ is a line $\ell\subset T_xM$, where $x=\pi(\ell)$. This
leads to a smooth subbundle $\Cal H\subset TN$, called the
\textit{tautological subbundle}. By definition, a tangent vector
$\xi\in T_\ell N$ lies in the subspace $\Cal H_\ell$ if and only if
$T_\ell\pi\cdot \xi\in\ell\subset T_xM$. By construction $\Cal
H\subset TN$ is a smooth subbundle of rank $m$, which contains the
vertical subbundle $V$ that has rank $m-1$. 

Now one defines a \textit{path geometry} on $M$ as a smooth line
subbundle $E\subset\Cal H\subset TN$, such that $\Cal H=E\oplus V$. As
a line bundle, $E$ is integrable and hence determines a foliation of
$N=\Cal PTM$ by 1--dimensional submanifolds. Since $E\cap V=\{0\}$, a
local integral manifold for $E$ always projects to a local
1--dimensional submanifold of $M$. Hence we really obtain a family of
paths in $M$. Moreover, taking the integral submanifold through
$\ell\in N$, the projection evidently passes through $x=\pi(\ell)$
with tangent space $\ell\subset T_xM$. Hence we see that in this
family there is exactly one path through each point in each direction.
It should be mentioned, that path geometry can be also interpreted as
describing the geometry of systems of second order ODE's, see
e.g.~\cite{C-tw} and \cite{C-two}.

In the spirit of filtered manifolds, one may go one step further, drop
the requirement that one deals with a projectivized tangent bundle and
just keep the configuration of subbundles with certain
(non--)integrability properties: Consider an arbitrary smooth manifold
$N$ of dimension $2m-1$ and two subbundles $E,V\subset TN$ of rank $1$
and $m-1$, such that $E\cap V=\{0\}$. Putting $\Cal H:=E\oplus V$, the
Lie bracket of vector fields induces a skew--symmetric bundle map
$\Cal H\x\Cal H\to TN/\Cal H$. Now the pair $(E,V)$ is said to define
a \textit{generalized path geometry} on $N$ if and only if this bundle
map vanishes on $V\x V$ and induces an isomorphism $E\otimes V\to
TN/\Cal H$. It is easy to see that this is always satisfied if $E$ and
$V$ come from a path geometry, see \cite{C-tw}. Further it turns out
that for $m\neq 3$, the subbundle $V$ in a generalized path geometry
is always involutive, and then the given geometry is locally
isomorphic to a path geometry on a local leaf space for the
corresponding foliation.

Any generalized path geometry on a manifold $N$ of dimension $2m-1$
induces a canonical normal parabolic geometry of type $(\hat G,\hat
P)$, where $\hat G=\tilde G=SL(m+1,\Bbb R)$ and $\hat P$ is the
subgroup of all elements which stabilize both the line spanned by the
first vector in the standard basis and the plane spanned by the first
two vectors in the standard basis of $\Bbb R^{m+1}$. On the level of
Lie algebras, we obtain a decomposition
$\hg=\hg_{-2}\oplus\hg_{-1}\oplus\hg_0\oplus\hg_1\oplus\hg_2$ such
that $\hp=\hg_0\oplus\hg_1\oplus\hg_2$ as well as decompositions
$\hg_{\pm 1}=\hg_{\pm 1}^E\oplus\hg_{\pm 1}^V$ according to the
following block decomposition with blocks of size $1$, $1$, and
$m-1$:
$$
\pmat{
\hat\g_0&\hat\g^E_1&\hat\g_2\\
\hat\g^E_{-1}&\hat\g_0&\hat\g^V_1\\
\hat\g_{-2}&\hat\g^V_{-1}&\hat\g_0}
$$
The subspaces $\hg_{-1}^E$ and $\hg_{-1}^V$ give rise to $\hat
P$--invariant subspaces in $\hg/\hp$ and the relation between the
parabolic geometry and the generalized path geometry is given by the
fact these two subspaces induce the subbundles $E$ and $V$ of $TN$, 
which define the generalized path geometry. Requiring the parabolic
geometry to be regular and to satisfy the normalization condition
discussed in \ref{3.2}, the parabolic geometry is uniquely determined
up to isomorphism. One obtains an equivalence of categories between
generalized path geometries and regular normal parabolic geometries in
this way.

\subsection{The path geometry of chains}\label{4.2}
Let $(M,H)$ be a contact manifold of dimension $2n+1$ endowed with a
contact projective structure, and let $(\G\to M,\om)$ be the
associated canonical Cartan geometry of type $(G,P)$ as in Theorem
\ref{2.2}. The chains of the contact projective structure can be
described as follows: Consider the one--dimensional subspace
$\fg_{-2}\subset\fg$ and the corresponding rank one subbundle
$\om^{-1}(\fg_{-2})\subset T\G$. This is involutive and since the
vertical subbundle corresponds to $\frak p\subset\fg$, local integral
submanifolds project to local one--dimensional immersed submanifolds
in $M$. Alternatively, the chains can be viewed as the projections of
the flow lines of the constant vector fields $\om^{-1}(X)$ with
$X\in\fg_{-2}$. This concept generalizes to all parabolic contact
structures. In that setting, it was shown in Section 4 of \cite{C-S-Z}
that chains are available through each point in $M$ tangent to each
line $\ell\in T_xM$ which is not contained in $H_x\subset T_xM$ and,
as an unparametrized curve, a chain is uniquely determined by its
tangent in one point.

This nicely fits into the picture of generalized path geometries. The
subset $\Cal P_0TM\subset\Cal PTM$ of lines not contained in the
contact distribution evidently is open, and it is a fiber bundle over
$M$ with fiber the complement of a hyperplane in $\Bbb RP^{2n}$. It is
also clear that the chains give rise to a generalized path geometry on
$\Cal P_0TM$. A description of this geometry in terms of $(\G\to
M,\om)$ can be found in Section 2.4 of \cite{CZ}. In our situation,
there is however a simple way to describe the parabolic geometry
corresponding to the path geometry of chains, at least in the case of
vanishing contact torsion. Namely, we know that the chains are
actually among the paths of the canonical subordinate projective
structure associated to the contact projective structure.

Applying the generalized Fefferman construction from \ref{3.1} to
$(\Cal G\to M,\om)$, we obtain a Cartan geometry $(\tcg\to
M,\tilde\om)$, which induces the subordinate projective structure on
$M$. To get the associated (generalized) path geometry, one applies
the correspondence space construction from \cite{C-tw}. By
construction, the subgroup $\hat P\subset \hat G=\tilde G$ is
contained in $\tilde P$. Hence one can form $N:=\tcg/\hat P$, which
can be identified with the total space of the fiber bundle
$\tcg\x_{\tilde P}(\tilde P/\hat P)$ over $M$. One immediately
verifies that $\tcg\x_{\tilde P}(\tilde P/\hat P)\cong\Cal PTM$. By
construction $(\tcg\to N,\tilde\om)$ is a Cartan geometry of type
$(\hat G,\hat P)$. In the case of vanishing contact torsion, $(\tcg\to
M,\tilde\om)$ is torsion free and normal. One easily verifies that
torsion freeness implies that the parabolic geometry $(\tcg\to
N,\tilde\om)$ is regular and by Proposition 2.4 of \cite{C-tw} it is
normal, too. Hence it is the canonical parabolic geometry associated
to the underlying path geometry, whose paths are the geodesics of the
connections in the projective class, see Section 4.7 of \cite{C-tw}.
Of course, the path geometry of chains can be recovered from this as
the restriction to the open subset $\Cal P_0TM\subset\Cal PTM=N$.
Having made these observations, the first part of the following result
is obvious, while the second essentially follows from the general
theory of correspondence spaces. 

\begin{prop*}
  Let $(M,H)$ be a contact manifold endowed with a contact projective
  structure.

  \noindent (1) There is a linear connection on $TM$ that has chains
  among its geodesics.

\noindent 
(2) If the given contact projective structure has vanishing contact
torsion, then the associated path geometry is torsion free if and only
if it is locally flat, which is equivalent to local flatness of the
initial contact projective structure. 
\end{prop*}
\begin{proof}
  (2) In Theorem \ref{T3.2} we have observed that local flatness of
  $(\G\to M,\om)$ is equivalent to local flatness of $(\tcg\to
  M,\tilde\om)$. Since $(\tcg\to M,\tilde\om)$ and $(\tcg\to
  N,\tilde\om)$ share the same curvature function, it is equivalent to
  local flatness of the latter geometry, too. Since $\Cal P_0TM\subset
  N$ is a dense open subset, we get the equivalence to local flatness
  of the path geometry of chains. Finally, it has been proved in
  Theorem 4.7 of \cite{C-tw} that for path geometries induced by
  projective structures torsion freeness implies local flatness.
\end{proof}

\begin{remark*}\label{R4.2}
  We have pointed out part (1) of this proposition only because it is
  in sharp contrast with the case of other parabolic contact
  structures. In \cite{CZ} it is shown that for integrable Lagrangean
  contact structures as well as CR structures of hypersurface type,
  the chains can \textit{never} be obtained as geodesics of a linear
  connection. Also part (2) is significantly different for those
  structures. While local flatness of the initial structure is
  equivalent to torsion freeness of the path geometry of chains, these
  path geometries are always non--flat for integrable Lagrangean
  contact and CR structures.
\end{remark*}

\subsection{Chain preserving contactomorphisms}\label{4.3}
As we have mentioned in the Remark above, for the parabolic contact
structures studied in \cite{CZ} the path geometry of chains is always
non--flat. It is proved there, that the parabolic contact structure
can essentially be recovered from the curvature of the path geometry of
chains. This leads to a conceptual proof of the fact that
contactomorphisms which map chains to chains (as unparametrized
curves) are homomorphisms (or anti--homomorphisms in an appropriate
sense) of the underlying parabolic contact structure. 

For contact projective structures, the situation is different of
course, since for a locally flat contact projective structure, also
the path geometry of chains is locally flat. Still we can show that,
assuming vanishing contact torsion, contactomorphisms which map chains
to chains are morphisms of contact projective structures.
\begin{thm*}\label{T4.3}
  For $i=1,2$ let $(M_i,H_i)$ be a contact manifolds endowed with
  contact projective structures with vanishing contact torsion. Let
  $f:M_1\to M_2$ be a contact diffeomorphism which maps chains to
  chains. Then $f$ is an isomorphism of contact projective structures.
\end{thm*}
\begin{proof}
  Put $N_i:=\Cal PTM_i$ and let $(p_i:\tcg_i\to N_i,\tilde\om_i)$ be
  the (regular normal) parabolic geometries associated to the path
  geometries determined by the subordinate projective structures.
  Consider $\Cal P_0TM_i\subset N_i$ and the restricted parabolic
  geometries $(p_i^{-1}(\Cal P_0TM_i)\to\Cal P_0TM_i,\tilde\om_i)$,
  which describe the path geometries of chains. By assumption, the
  contactomorphism $f$ induces a morphism of these path geometries,
  so it lifts to a morphism $\Psi$ between the Cartan geometries.
  
  Our aim is to extend $\Psi$ to a morphism between the Cartan
  geometries $(\tcg_i\to N_i,\tilde\om_i)$. Choose a local smooth
  section $\si$ of the principal bundle $\tilde p_1:\tcg_1\to M_1$,
  which has values in $p_1^{-1}(\Cal P_0TM_1)$, and let $U\subset M_1$
  be its domain of definition. Then there is a unique $\tilde
  P$--equivariant map $\Psi_\si:\tilde p_1^{-1}(U)\to\tcg_2$ such that
  $\Psi_\si(\si(x))=\Psi(\si(x))$ for any $x\in U$. We claim that
  $\Psi_\si$ is an extension of $\Psi$ to $\tilde p_1^{-1}(U)$.  To
  prove this take, a point $x\in U$ and consider
$$ 
A_x := \{u\in p_1^{-1}(x): \Psi(u)=\Psi_\si(u)\}\subseteq\tilde p_1^{-1}(x). 
$$
By definition $\si(x)\in A_x$, so this set is non--empty. Further,
since both $\Psi$ and $\Psi_\si$ are equivariant for the principal
right action of $\hat P\subset\tilde P$, the set $A_x$ is $\hat
P$--invariant, and it is closed by definition.

For $A\in\tp$, the fundamental vector fields $\ze_A^i\in\frak
X(\tcg_i)$ are given by $\tilde\om_i^{-1}(A)$. Since
$\Psi^*\tilde\om_2=\tilde\om_1$, we conclude that
$T\Psi\o\ze_A^1=\ze_A^2\o\Psi$, so $\Psi$ also intertwines the flows
of these vector fields, whenever they are defined. Otherwise put, for
any $u\in p_1^{-1}(\Cal P_0TM_1)$ there is a neighbourhood $V$ of
$e\in\tilde P$ such that $\Psi(ug)=\Psi(u)g$ for all $g\in V$. Since
$\Psi_\si$ is $\tilde P$--equivariant by definition, this implies that
for any $u\in A_x$ a neighborhood of $u$ is contained in $A_x$, so
$A_x$ is open. Since we have noted above that $A_x$ is $\hat
P$--equivariant, we can prove that $A_x=p_1^{-1}(x)$ and hence our
claim by showing that the image of $A_x$ surjects onto $\Cal
P_0T_xM_1\subset\Cal PTM_1$. But the projection to $\Cal P_0T_xM_1$ is
a surjective submersion and hence an open mapping. Since both $A_x$
and its complement are open, also the image of $A_x$ in $\Cal
P_0T_xM_1$ is open and closed. Since $\Cal P_0T_xM_1$ is the
complement of a hyperplane in projective space and hence connected,
the proof of the claim is complete.

By construction, $\Psi_\si:\tilde p^{-1}(U)\to\tilde\G_2$ covers
$f|_U:U\to M_2$, so we can view it as a morphism between the Cartan
geometries $(\tilde p_1^{-1}(U),\tilde\om_1)$ and $(\tilde
p_2^{-1}(f(U)),\tilde\om_2)$. But this exactly means that $f|_U$ is a
morphism between the subordinate projective structures, so in
particular it locally preserves the contact geodesics. Hence locally
and thus globally $f$ is a morphism of contact projective structures. 
\end{proof}

\end{document}